\documentclass{amsart}
\usepackage{latexsym, amssymb, amsmath}

\newtheorem{conj}{Conjecture}
\newtheorem{thm}{Theorem}[section]
\newtheorem{lemm}[thm]{Lemma}
\newtheorem{cor}[thm]{Corollary}
\newtheorem{prop}[thm]{Proposition}
\newtheorem{rmk}[thm]{Remark}
\theoremstyle{definition}
\newtheorem{defi}[thm]{Definition}

\newcommand{\lapla}{\Delta}

\newcommand{\p}{\phi}

\newcommand{\met}{\langle \cdot , \cdot \rangle}

\title{Biminimal properly immersed submanifolds \\
in complete Riemannian manifolds\\
 of non-positive curvature}

\author{Shun Maeta}
\thanks{supported by 
Research Fellowships of the Japan Society for the Promotion of Science for Young Scientists, No.~23-6949.}
\address{\footnotesize{Division of Mathematics, Graduate School of Information Sciences,
 Tohoku University, Sendai 980-8579, Japan.}
 }
\footnotesize{
\email{maeta@ims.is.tohoku.ac.jp~{\it or}~shun.maeta@gmail.com}
}

\begin{document} 
\maketitle 
\markboth{Biminimal properly immersed submanifolds} 
{Shun Maeta}

\begin{abstract} 
We consider a non-negative biminimal properly immersed submanifold $M$ (that is, a biminimal properly immersed submanifold with $\lambda\geq0$) in a complete Riemannian manifold $N$ with non-positive sectional curvature.
 Assume that the sectional curvature $K^N$ of $N$ satisfies $K^N\geq-L(1+{\rm dist}_N(\cdot, q_0)^2)^{\frac{\alpha}{2}}$ for some $L>0,$ $2>\alpha \geq 0$ and $q_0\in N$.
   Then, we prove that $M$ is minimal.
As a corollary, we give that any biharmonic properly immersed submanifold in a hyperbolic space is minimal.
 These results give affirmative partial answers to the global version of generalized Chen's conjecture.
\end{abstract}

\vspace{10pt}

%%%%%%%%%%%%%%%%%%%%%%%%%%%%%%%%%%%%%%%%%%%

\section{\bf Introduction}\label{intro} 
Theory of harmonic maps has been applied into various fields in differential geometry.
 Harmonic maps between two Riemannian manifolds are
 critical points of the energy functional $E(\p)=\frac{1}{2}\int_M\|d\p\|^2dv_g$, for smooth maps $\p:(M^m,g)\rightarrow (N^n,h)$ from an $m$-dimensional Riemannian manifold into an $n$-dimensional Riemannian manifold, 
 where $dv_g$ denotes the volume element of $g.$

On the other hand, in 1983, J. Eells and L. Lemaire \cite{jell1} proposed the problem to consider
 {\em polyharmonic maps of order k}.
In 1986, G.Y. Jiang \cite{jg1} studied {\em biharmonic maps} (that is, polyharmonic maps of order $2$) which are critical points of the {\em bi-energy}
 \begin{equation}
 E_2(\p):=\int_M |\tau(\p)|^2 dv_g,
 \end{equation} 
 where $\tau(\p)$  denotes the {\em tension field} of $\p.$
The Euler-Lagrange equation of $E_2$ is 
$$\tau_2(\p):=-\lapla^\p\tau(\p)-\sum^m_{i=1}R^{N}(\tau(\p),d\p(e_i))d\p(e_i)=0,$$
where 
$\lapla^{\p}:=\sum^m_{i=1}(\nabla^{\p}_{e_i}\nabla^{\p}_{e_i}-\nabla^{\p}_{\nabla_{e_i}e_i})$, 
$R^{N}$ is the Riemannian curvature of $N$ i.e., $R^{N}(X,Y)Z:=[\nabla^N_{X},\nabla^N_Y]Z-\nabla^N_{[X,Y]}Z$ for any vector field X, Y and Z on $N$,
 and $\{e_i\}_{i=1}^m$ is a local orthonormal frame field on $M$.

If an isometric immersion $\p:(M,g)\rightarrow (N,h)$ is biharmonic,
 then $M$ is called a {\em biharmonic submanifold} in $N$.
 In this case, we remark that the tension field $\tau(\p)$ of $\p$ is written as $\tau(\p)=m{\bf H}$, where ${\bf H}$ is the mean curvature vector field of $M$.

For biharmonic submanifolds, 
there is an interesting problem, namely Chen's conjecture 
(cf.~\cite{Chen}).

\begin{conj} \label{Chen}
Any biharmonic submanifold in $\mathbb{E}^n$ is minimal. 
\end{conj} 

There are many affirmative partial answers to Conjecture~$\ref{Chen}$ 
(cf.~\cite{Chen}, \cite{Chen-Ishikawa-2}, \cite{Leuven}, \cite{Dimi}, \cite{Hasanis-Vlachos}). 
Conjecture $\ref{Chen}$ is solved completely 
if $M$ is one of the following: 
(a) a curve \cite{Dimi}, 
(b) a surface in $\mathbb{E}^3$ \cite{Chen}, 
(c) a hypersurface in $\mathbb{E}^4$ \cite{Leuven}, \cite{Hasanis-Vlachos}. 

Note that, 
since there is no assumption of {\it completeness} for submanifolds in Conjecture~1, 
in a sense it is a problem in {\it local} differential geometry.  
Recently, Conjecture~1 was reformulated  into a problem 
in {\it global} differential geometry as follows (cf.~\cite{kasm1},~\cite{sm7},~\cite{N-U-1},~\cite{N-U-2}):

\begin{conj} \label{Akutagawa Maeta}
Any {\rm complete} biharmonic submanifold in $\mathbb{E}^n$ is minimal. 
\end{conj} 

On the other hand, Conjecture $\ref{Chen}$ was generalized as follows:
{\em Any biharmonic submanifold in a Riemannian manifold with non-positive sectional curvature is minimal}
(cf. \cite{absmco1}, \cite{absmco2}, \cite{rcsmco1}, \cite{rcsmco2}, \cite{rcsmpp1}). 
 This generalization is also a problem in local differential geometry.
 Y.-L. Ou and L. Tang \cite{ylolt1} gave a counterexample of this conjecture 
 (see~\cite{jg1} for an affirmative answer).
With these understandings, it is natural to consider the following conjecture.

\begin{conj}\label{glo gen Chen}
Any complete biharmonic submanifold in a Riemannian manifold with non-positive sectional curvature is minimal. 
\end{conj}

N. Nakauchi and H. Urakawa gave an affirmative partial answer to Conjecture $\ref{glo gen Chen}$ (cf. \cite{N-U-1}, \cite{N-U-2}).

\vspace{10pt}

An immersed submanifold $M$ in a Riemannian manifold $N$ is said to be {\em properly immersed}
 if the immersion is a proper map.
  K. Akutagawa and the author gave an affirmative partial answer to
 Conjecture $\ref{Chen}$ (Conjecture $\ref{Akutagawa Maeta}$ particularly) as follows (cf.~\cite{kasm1},~\cite{sm7}):
\begin{thm}[\cite{kasm1}]
Any biharmonic properly immersed submanifold in $\mathbb{E}^n$ is minimal.
\end{thm}
For Conjecture $\ref{glo gen Chen}$, we consider a biharmonic properly immersed submanifold in a complete Riemannian manifold with non-positive sectional curvature.

\vspace{10pt}

Recently, E. Loubeau and S. Montaldo \cite{elsm1} introduced a {\em biminimal immersion} as follows:

\begin{defi}[\cite{elsm1}]
An immersion $\p:(M^m,g)\rightarrow (N^n,h)$, $m\leq n$ is called {\em biminimal} if it is a critical point of the functional 
$$E_{2,\lambda}(\p)=E_2(\p)+\lambda E(\p),\ \ \lambda\in \mathbb{R},$$
 for any smooth variation of the map $\p_t (-\varepsilon<t<\varepsilon)$, $\p_0=\p$ such that $\left. V=\frac{d\p_t}{dt}\right |_{t=0}$ is normal to $\p(M)$.
\end{defi}

The Euler-Lagrange equation of $E_{2,\lambda}$ is
$$[\tau_2(\p)]^{\perp}+\lambda[\tau(\p)]^{\perp}=0,$$
where $[\cdot]^{\perp}$ denotes the normal component of $[\cdot]$.
We call an immersion {\em free biminimal} if it is biminimal condition for $\lambda=0$.
 If $\p:(M,g)\rightarrow (N,h)$ is an isometric immersion, then the biminimal condition is
\begin{align}\label{biminimal eq}
\left[-\lapla^{\p}{\bf H}-\sum^m_{i=1}R^{N}({\bf H},d\p(e_i))d\p(e_i)\right]^{\perp}+\lambda{\bf H}=0,
\end{align}
for some $\lambda \in\mathbb{R}$, and then $M$ is called a {\em biminimal submanifold} in $N$.
 If $M$ is a biminimal submanifold with $\lambda \geq0$ in $N$, then $M$ is called a {\em non-negative biminimal submanifold} in $N$. 
\begin{rmk}\label{biharmonic is biminimal}
We remark that every biharmonic submanifold is free biminimal one.
\end{rmk}

 Before mentioning our main theorem, we define the following notion.

\begin{defi}
For a complete Riemannian manifold $(N,h)$ and $\alpha \geq 0,$
if the sectional curvature $K^N$ of $N$ satisfies
$$K^N \geq-L(1+{\rm dist}_N(\cdot,q_0)^2)^{\frac{\alpha}{2}},~~{\rm for~some }~L>0~{\rm and}~q_0\in N.$$
Then we shall call that $K^N$ has a {\em polynomial growth bound of order} $\alpha$ {\em from below.}
\end{defi}
In this article, our main theorem is the following.
 \begin{thm}\label{main theorem}
Let $(N,h)$ be a complete Riemannian manifold with non-positive sectional curvature. Assume that the sectional curvature $K^N$ has a 
polynomial growth bound of order less than $2$ from below.
Then, any non-negative biminimal properly immersed submanifold in $N$ is minimal.
 \end{thm}
Since every biharmonic submanifold is free biminimal one, we obtain the following result.

\begin{cor}\label{main cor}
Let $(N,h)$ be a complete Riemannian manifold with non-positive sectional curvature. Assume that the sectional curvature $K^N$ has a 
polynomial growth bound of order less than $2$ from below.
Then, any biharmonic properly immersed submanifold in $N$ is minimal.
\end{cor}
This result gives an affirmative partial answer to Conjecture \ref{glo gen Chen}.

\begin{rmk}
If $N$ is a complete Riemannian manifold whose non-positive sectional curvature is bounded from below $($including a hyperbolic space$),$
 then $N$ satisfies the assumption in Corollary $\ref{main cor}$.
 
 For the case of $\lambda<0$, the author constructed non-minimal biminimal submanifolds
 in $\mathbb{E}^n$ $($cf. \cite{sm7}$).$
\end{rmk}

The remaining sections are organized as follows. 
Section $\ref{Pre}$ contains some necessary definitions and preliminary geometric results.
 In section $\ref{EC}$, we prove our main theorem.
 In section $\ref{Ein}$, we show that any complete biharmonic submanifold
 with Ricci curvature bounded from below 
 in a Riemannian manifold with non-positive sectional curvature is minimal.
In section $\ref{hyp}$, we consider a non-negative biminimal hypersurface in a Riemannian manifold.

\qquad\\
\qquad\\

%%%%%%%%%%%%%%%%%%%%%%%%%%%%%%%%%%%%%%%%%%%%%%%%%%%%%%%%%%%%%%%%%%%%
\section{\bf Preliminaries}\label{Pre} 

Let $\p:(M^m,g)\rightarrow (N^n,h=\met)$ be an isometric immersion from an $m$-dimensional Riemannian manifold into an $n$-dimensional Riemannian manifold.
In this case, we identify $d\p(X)$ with $X\in \frak{X}(M)$ for each $x\in M.$
We also denote by $\met$ the induced metric $\p^{-1}h$.

Then, the Gauss formula is given by
\begin{equation}
\nabla^N_XY=\nabla _XY+B(X,Y),\ \ \ \ X,Y\in \frak{X}(M),
\end{equation}
where $\nabla^N$ and $\nabla$ are the Levi-Civita connections on $N$ and $M$ respectively, and $B$ is the second fundamental form of $M$ in $N$.
The Weingarten formula is given by
\begin{equation}\label{Wformula}
\nabla^N_X \xi =-A_{\xi}X+\nabla^{\perp}_X{\xi},\ \ \ X\in \frak{X}(M),\  \xi \in \frak{X}(M)^{\perp},  
\end{equation}
where $A_{\xi}$ is the shape operator for a normal vector field $\xi$ on $M,$ and $\nabla^{\perp}$ denotes the normal connection of the normal bundle on $M$ in $N$.
It is well known that $B$ and $A$ are related by
\begin{equation}\label{BA rel}
\langle B(X,Y), \xi \rangle=\langle A_{\xi}X,Y \rangle.
\end{equation}

For any $x \in M$, 
let $\{e_1, \cdots, e_m, e_{m+1}, \cdots, e_n\}$ be an orthonormal basis of $N$ at $x$ 
such that $\{e_1, \cdots, e_m\}$ is an orthonormal basis of $T_xM$. 
Then, $B$ is decomposed as 
$$ 
B(X, Y) = \sum_{\alpha=m+1}^n B_{\alpha}(X, Y)e_{\alpha},~~{\rm at}~x. 
$$ 
The mean curvature vector field ${\bf H}$ of $M$ at $x$ is also given by 
$$ 
{\bf H}(x) = \frac{1}{m} \sum_{i = 1}^m B(e_i, e_i) =\sum_{\alpha=m+1}^n H_{\alpha}(x)e_{\alpha},\qquad 
H_{\alpha}(x) := \frac{1}{m} \sum_{i = 1}^m B_{\alpha}(e_i, e_i).  
$$ 

The necessary and sufficient condition for $M$ in $N$ 
to be biharmonic is the following (cf. \cite{smhu1}): 
\begin{align}
\ \ \Delta^{\perp} {\bf H} - \sum_{i=1}^m B(A_{\bf H}e_i, e_i) +\left[\sum^m_{i=1} R^N( {\bf H} , d\p(e_i))d\p(e_i)\right]^{\perp} = 0, \\ 
\ \ m~\nabla |{\bf H}|^2 + 4~{\rm trace}~A_{\nabla^{\perp} {\bf H}} +\left[\sum^m_{i=1}R^N({\bf H},d\p(e_i))d\p(e_i)\right]^T= 0,
\end{align} 
where $\Delta^{\perp}$ is the (non-positive) Laplace operator associated with the normal connection $\nabla^{\perp}$. 
Similarly, the necessary and sufficient condition for $M$ in $N$ to be biminimal is the following:
\begin{align}\label{N-S biminimal}
\ \ \Delta^{\perp} {\bf H} - \sum_{i=1}^m B(A_{\bf H}e_i, e_i) +\left[\sum^m_{i=1} R^N( {\bf H} , d\p(e_i))d\p(e_i)\right]^{\perp} = \lambda{\bf H}.
\end{align}

\qquad\\
\qquad\\

%%%%%%%%%%%%%%%%%%%%%%%%%%%%%%%%%%%%%%%%%%%%%%%%%%%%%%%%%%

\section{Proof of main theorem}\label{EC} 
In this section, we prove Theorem \ref{main theorem}.
 We shall show the following lemma.

 \begin{lemm}\label{proper sub mfd minimal}
Let $(M,g)$ be a properly immersed submanifold in a complete Riemannian manifold $(N,h)$
 whose sectional curvature $K^N$ has a polynomial growth bound of order less than $2$ from below. 
Assume that there exists a positive constant $k > 0$ such that
 \begin{equation}\label{key assumption}
 \Delta |{\bf H}|^2\geq k |{\bf H}|^4 \ \ \text{on}\ \  M.
 \end{equation} 
 Then $M$ is minimal.
\end{lemm}

\begin{proof}
If $M$ is compact, applying the standard maximum principle to the elliptic inequality (\ref{key assumption}), 
we have that ${\bf H} = 0$ on $M$. 
Therefore, we may assume that $M$ is noncompact. 
Suppose that ${\bf H}(x_0) \ne 0$ at some point $x_0 \in M$. 
Then, we will lead a contradiction. 

Set 
$$ 
u(x) := |{\bf H}(x)|^2\quad {\rm for}\ \ x \in M. 
$$  
For each $\rho > 0$, consider the function 
$$ 
f(x) = f_{\rho}(x) := (\rho^2 - r(\p(x))^2)^2 u(x)\quad {\rm for}\ \ 
x \in M\cap \p^{-1}(\overline{{\bf B}_{\rho}}), 
$$ 
where $r(\p(x))={\rm dist}_N (\p(x),q_0)$ for some $q_0\in N$ and $\overline{{\bf B}_{\rho}}:=\{q\in N\left|r(q)\leq \rho\right.\}$ denote respectively the geodesic distance from $q_0$ and the closed geodesic ball of radius $\rho$ centered at $q_0$.
Then, there exists $\rho_0>0$ such that $x_0\in M \cap \p^{-1}({\bf B}_{\rho_0})$.
 For each $\rho\geq \rho_0$, $f$ is a non-negative function which is not identically zero on $M\cap \p^{-1}(\overline{\bf B}_{\rho})$.
 Take any $\rho \geq \rho_0$ and fix it.
  Since $M$ is properly immersed in $N$, $M\cap \p^{-1}(\overline{\bf B}_{\rho})$ is compact. 
 By this fact combined with $f=0$ on $M\cap\p^{-1}({\partial \overline {\bf B}_{\rho}})$, there exists a maximum point $p\in M\cap \p^{-1}({\bf B}_{\rho})$ of $f$ such that $f(p) > 0$.
 
 We consider the case that $\p(p)$ is not on the cut locus of $q_0$.
 We have $\nabla f= 0$ at $p$, and hence 
\begin{equation}\label{grad}
\frac{\nabla u(p)}{u(p)} = \frac{2~\nabla r(\p(p))^2}{\rho^2 - r(\p(p))^2}. 
\end{equation} 
We also have that $\Delta f \leq 0$ at $p$. 
Combining this with (\ref{grad}), we obtain 
\begin{equation}\label{Lap} 
\frac{\Delta u(p)}{u(p)} \leq 
\frac{6|\nabla r(\p(p))^2|^2}{(\rho^2-r(\p(p))^2)^2}
+\frac{2\Delta r(\p(p))^2}{\rho^2-r(\p(p))^2}.
\end{equation}   
By a direct computation, we have
\begin{equation}\label{est nab}
|\nabla r(\p(p))^2|^2\leq 4m r(\p(p))^2,
\end{equation}
and 
\begin{equation}\label{est Lap}
\begin{aligned}
\Delta r(\p(p))^2
=&~2\sum^m_{i=1}\langle (\nabla r) \p(p), d\p(e_i)\rangle^2\\
&+2 r (\p(p)) \sum^m_{i=1}D^2r(\p(p))(d\p(e_i),d\p(e_i))\\
&+2 r (\p(p)) \langle (\nabla r) (\p(p)),\tau(\p)(p)\rangle\\ 
\leq &~2m+2 r (\p(p))\sum^m_{i=1}D^2r(\p(p))(d\p(e_i),d\p(e_i))\\
&+2mr (\p(p)) |{\bf H}(p)|,
\end{aligned}
\end{equation}
where $D^2r$ denotes the {\em Hessian} of $r$.
Since the sectional curvature of $N$ satisfies
$K^N \geq -L(1+r^2)^{\frac{\alpha}{2}} \geq -L(1+\rho^2)^{\frac{\alpha}{2}}$ on $\overline {\bf B}_{\rho}$,
by an elementary argument,
 we obtain the following (cf. \cite{t.sakai1}),
\begin{equation}\label{est Hes}
\begin{aligned}
&\sum^m_{i=1}D^2r(\p(p))(d\p(e_i),d\p(e_i))\\
&\hspace{70pt}\leq m\sqrt{L(1+\rho^2)^{\frac{\alpha}{2}}}~ \text{coth} \left(\sqrt{L(1+\rho^2)^{\frac{\alpha}{2}}} r(\p(p))\right).
\end{aligned}
\end{equation}
 Combining $(\ref{est Lap})$ and $(\ref{est Hes})$, we have
\begin{equation}\label{est Lap2}
\begin{aligned}
\Delta r(\p(p))^2
\leq &~2m+2m\sqrt{L(1+\rho^2)^{\frac{\alpha}{2}}} r(\p(p)) \text{coth} \left(\sqrt{L(1+\rho^2)^{\frac{\alpha}{2}}} r(\p(p))\right) \\
&+2mr(\p(p)) |{\bf H}(p)|.
\end{aligned}
\end{equation}
It follows from $(\ref{key assumption})$, $(\ref{Lap})$, $(\ref{est nab})$ and $(\ref{est Lap2})$ that
 \begin{equation*}
 \begin{aligned}
ku(p)\leq &~ \frac{24mr(\p(p))^2}{(\rho^2-r(\p(p))^2)^2}\\
 &+\frac{4m+4m\sqrt{L(1+\rho^2)^{\frac{\alpha}{2}}} r(\p(p)) \text{coth} \left(\sqrt{L(1+\rho^2)^{\frac{\alpha}{2}}} r(\p(p))\right) +4mr(\p(p))|{\bf H}(p)|}{\rho^2-r(\p(p))^2},
 \end{aligned}
 \end{equation*}
   and hence
\begin{align*}
kf(p) \leq &~24mr(\p(p))^2+4m(\rho^2-r(\p(p))^2)\\
&+4m\sqrt{L(1+\rho^2)^{\frac{\alpha}{2}}} r(\p(p)) \text{coth} \left(\sqrt{L(1+\rho^2)^{\frac{\alpha}{2}}} r(\p(p))\right) (\rho^2-r(\p(p))^2)\\
&+4mr(\p(p))\sqrt{f(p)}.
\end{align*}
By an elementary argument, we only have to consider the following:
 there exists a positive constant $c$ depending only on $k,L$ and $m$ such that
$$f(p)\leq c (1+\rho^2)^{\frac{\alpha+6}{4}}.$$
Since $f(p)$ is the maximum of $f$, we have 
$$ 
f(x) \leq f(p) \leq c (1+\rho^2)^{\frac{\alpha+6}{4}} \quad
 {\rm for}\ \ x \in M \cap \p^{-1}\big{(} \overline{{\bf B}_{\rho}} \big{)}, 
$$ 
and hence 
\begin{equation}\label{Final} 
|{\bf H}(x)|^2 = u(x) \leq \frac{c (1+\rho^2)^{\frac{\alpha+6}{4}}}{(\rho^2 - r(\p(x))^2)^2}\quad 
{\rm for}\ \ x \in M \cap \p^{-1}\big{(} {\bf B}_{\rho} \big{)}, \quad {\rm and}\ \ \rho \geq \rho_0. 
\end{equation} 
Letting $\rho \nearrow \infty$ in (\ref{Final}) for $x = x_0$, 
we have that 
$$ 
|{\bf H}(x_0)|^2 = 0, 
$$  
because of the assumption $\alpha<2.$
This contradicts our assumption that ${\bf H}(x_0) \ne 0$.  
Therefore $M$ is minimal. 

If $\p(p)$ is on the cut locus of $q_0$, then we use a meted of Calabi (cf. \cite{Calabi}).
 Let $\sigma$ be a minimal geodesic joining $\p(p)$ and $q_0$.
 Then for any point $q'$ in the interior of $\sigma$, $q'$ is not conjugate to $q_0.$
 Fix such a point $q'$.
 Let $U_{q'}\subset {\bf B}_{\rho}$ be a conical neighborhood of the geodesic segment of $\sigma$ joining $q'$
 and $\p(p)$ such that, for any $\p(x)\in U_{q'}$, there is at most one minimizing geodesic joining $q'$ and $\p(x).$
 Let $\bar r(\p(x))={\rm dist}_{U_{q'}}(\p(x),q')$ in the manifold $U_{q'}.$
 Then we have ${\bar r}(\p(x))\geq {\rm dist}_{N}(\p(x),q')$, $r(\p(x))\leq r(q')+{\bar r}(\p(x))$, $r(q')+{\bar r}(\p(p))=r(\p(p))$ 
 and ${\bar r}$ is smooth in a neighborhood of $\p(p).$ We claim that the function
$$ 
{\bar f}(x) = {\bar f}_{\rho}(x) := (\rho^2 - \{r(q')+{\bar r}(\p(x))\}^2)^2 u(x)\quad {\rm for}\ \ 
x \in M\cap \p^{-1}(U_{q'})
$$ 
 also attains a local maximum at the point $p.$ In fact, for any point $x\in M\cap \p^{-1}(U_{q'}),$
 we have
 \begin{align*}
 {\bar f}(p)
 =&(\rho^2 - \{r(q')+{\bar r}(\p(p))\}^2)^2 u(p)\\
 =&(\rho^2 - r(\p(p))^2)^2 u(p)\\
 =&f(p)\geq f(x)\\
 =& (\rho^2 - r(\p(x))^2)^2 u(x)\\
 \geq & (\rho^2 - \{r(q')+{\bar r}(\p(x))\}^2)^2 u(x)\\
 =& {\bar f}(x).
 \end{align*}
Therefore the claim is proved and we can take the gradient and the Laplacian of the function ${\bar f}(x)$ at $p.$
The same argument as before then shows that 
\begin{align*}
k{\bar f}(p) \leq &~24mr(\p(p))^2+4m(\rho^2-r(\p(p))^2)\\
&+4m\sqrt{L(1+\rho^2)^{\frac{\alpha}{2}}} r(\p(p)) \text{coth} \left(\sqrt{L(1+\rho^2)^{\frac{\alpha}{2}}} {\bar r}(\p(p))\right) (\rho^2-r(\p(p))^2)\\
&+4mr(\p(p))\sqrt{{\bar f}(p)}.
\end{align*}
Take $q'\rightarrow q_0.$
By an elementary argument, we only have to consider the following:
 there exists a positive constant $c$ depending only on $k,L$ and $m$ such that
$$f(p)={\bar f}(p)\leq c (1+\rho^2)^{\frac{\alpha+6}{4}}.$$
The same argument as before then shows that $M$ is minimal. 
 \end{proof}

 By the same argument as in Lemma $\ref{proper sub mfd minimal}$,
 we also obtain the following results.

 \begin{prop}
Let $(M,g)$ be a properly immersed submanifold in a complete Riemannian manifold $(N,h)$ whose sectional curvature $K^N$ has a polynomial growth bound of order less than $2$ from below.
Assume that there exists a positive constant $k > 0$ such that
\begin{equation}\label{key3 assumption}
\Delta |B|^2 \geq k |B|^4 \quad {\rm on}\ \ M,
\end{equation} 
where $|B|$ is the norm of the second fundamental form.
Then $M$ is totally geodesic.
\end{prop}

\begin{proof}
In general, we have $m|{\bf H}|^2 \leq |B|^2$. 
 By using this inequality, the same argument as in Lemma $\ref{proper sub mfd minimal}$ shows the proposition.
\end{proof}

\begin{prop}
Let $(M,g)$ be a properly immersed submanifold in a complete Riemannian manifold $(N,h)$ whose sectional curvature $K^N$ has a polynomial growth bound of order less than $2$ from below.
Let $u$ be a smooth non-negative function on $M$. 
Assume that there exists a positive constant $k > 0$ such that
\begin{equation}\label{key3 assumption}
\Delta u \geq k u^2\quad {\rm on}\ \ M.
\end{equation} 
If the mean curvature is bounded from above by a constant $C$,
then $u = 0$ on $M$.
\end{prop}

\begin{proof}
The same argument as in Lemma $\ref{proper sub mfd minimal}$ shows the proposition.
\end{proof}

From the equation of (\ref{N-S biminimal}), we obtain the following lemma.

 \begin{lemm} \label{key lemma}
Let $(M,g)$ be a non-negative biminimal submanifold $($that is, a biminimal submanifold with $\lambda \geq0$ $)$ in a Riemannian manifold with non-positive sectional curvature. 
Then, the following inequality for $|{\bf H}|^2$ holds 
\begin{equation}\label{key} 
\Delta |{\bf H}|^2 \geq 2m~|{\bf H}|^4. 
\end{equation} 
\end{lemm}

\begin{proof} 
The equation of (\ref{N-S biminimal}) implies that, at each $x \in M$, 
\begin{equation}\label{ell-ineq}
\begin{aligned} 
\Delta |{\bf H}|^2 
 =&~2~\sum_{i=1}^m \langle \nabla_{e_i}^{\perp} {\bf H}, \nabla_{e_i}^{\perp} {\bf H} \rangle
+ 2~\langle \Delta^{\perp} {\bf H}, {\bf H} \rangle  \\ 
 =&~2~\sum_{i=1}^m \langle \nabla_{e_i}^{\perp} {\bf H}, \nabla_{e_i}^{\perp} {\bf H} \rangle 
+2~\sum_{i=1}^m \langle B(A_{\bf H} e_i, e_i), {\bf H} \rangle\\
&-2~\left\langle \sum^m_{i=1} R^N ( {\bf H} , d\p(e_i) ) d\p(e_i), {\bf H} \right \rangle
+ 2~\lambda \langle {\bf H}, {\bf H} \rangle  \\  
 \geq &~2~\sum_{i=1}^m \langle A_{\bf H} e_i, A_{\bf H} e_i \rangle .
\end{aligned}
\end{equation} 
The last inequality follows from $K^N\leq 0$, $\lambda \geq0$ 
 and $(\ref{BA rel})$.
When ${\bf H}(x) \ne 0$, set $e_n := \frac{{\bf H}(x)}{|{\bf H}(x)|}$. 
Then, ${\bf H}(x) = H_n(x) e_n$ and $|{\bf H}(x)|^2 = H_n(x)^2$. 
From (\ref{ell-ineq}), we have at $x$ 
\begin{align} 
\Delta |{\bf H}|^2 
& \geq 2~H_n^2~\sum_{i=1}^m \langle A_{e_n} e_i, A_{e_n} e_i \rangle\notag \\ 
& = 2~|{\bf H}|^2~|B_n|_g^2 \\
& \geq 2m~|{\bf H}|^4 \notag. 
\end{align} 
Even when ${\bf H}(x) = 0$, the above inequality~(\ref{key}) still holds at $x$. 
This completes the proof. 
\end{proof} 

From the inequality $(\ref{key})$, we obtain the following propositions.

\begin{prop}
Let $(M,g)$ be a compact non-negative biminimal submanifold in a Riemannian manifold with non-positive sectional curvature. 
Then, $M$ is minimal.
\end{prop}

\begin{proof}
 Applying the standard maximum principle to the elliptic inequality (\ref{key}), 
we have that ${\bf H} = 0$ on $M$. 
\end{proof}

\begin{prop}
Let $(M,g)$ be a non-negative biminimal submanifold in a Riemannian manifold with non-positive sectional curvature. 
If the mean curvature is constant, then $M$ is minimal.
\end{prop}

\begin{proof}
Since $\Delta |{\bf H}|^2=0$, by using $(\ref{key})$, we obtain the proposition.  
\end{proof}

We shall show our main theorem (cf. Theorem $\ref{main theorem}$).

\begin{proof}[Proof of Theorem $\ref{main theorem}$]\quad
By using Lemma $\ref{key lemma}$, we obtain the inequality $(\ref{key})$.
Therefore, by using Lemma $\ref{proper sub mfd minimal}$, we obtain Theorem $\ref{main theorem}.$ 
\end{proof}

\qquad\\
\qquad\\

%%%%%%%%%%%%%%%%%%%%%%%%%%%%%%%%%%%%%%%%%%%%%%%%%%%%%%%%%%%%%
\section{Another result for Conjecture $\ref{glo gen Chen}$}\label{Ein}
In this section, we show that any complete biharmonic submanifold with Ricci curvature
 bounded from below in a Riemannian manifold with non-positive sectional curvature is minimal.

We recall the generalized maximum principle developed in Cheng-Yau \cite{Cheng-Yau}.
\begin{lemm}[\cite{Cheng-Yau}]\label{Cheng-Yau}
Let $(M, g)$ be a complete Riemannian manifold
whose Ricci curvature is bounded from below. 
Let $u$ be a smooth non-negative function on $M$. 
Assume that there exists a positive constant $k > 0$ such that
\begin{equation}\label{GMP}
\Delta u \geq k u^2\quad {\rm on}\ \ M.
\end{equation} 
Then, $u = 0$ on $M$.
\end{lemm}

By using Lemma $\ref{Cheng-Yau}$ and the inequality $(\ref{key})$, we obtain the following proposition.

\begin{prop}\label{Ric below}
Let $(M,g)$ be a complete non-negative biminimal submanifold in a Riemannian manifold with
 non-positive sectional curvature.
 If the Ricci curvature of $M$ is bounded from below,
 then $M$ is minimal.
 \end{prop}

Proposition \ref{Ric below} implies the following result.

\begin{cor}\label{Ric below cor}
Let $(M,g)$ be a complete biharmonic submanifold in a Riemannian manifold with
 non-positive sectional curvature.
 If the Ricci curvature of $M$ is bounded from below, 
 then $M$ is minimal.

\end{cor}

This result gives an affirmative partial answer to Conjecture \ref{glo gen Chen}.

\qquad\\
\qquad\\

%%%%%%%%%%%%%%%%%%%%%%%%%%%%%%%%%%%%%%%%%%%%%%%%%%%%%
\section{Non-negative biminimal hypersurfaces in  Riemannian manifolds}\label{hyp}

In this section, we consider a non-negative biminimal hypersurface $(M^m,g)$
 in a Riemannian manifold $(N^{m+1},h)$.
In this case, we can denote that ${\bf H}=H\xi$, where $H$ and $\xi$ are the mean curvature and
 a unit normal vector field along $\p$ respectively.
We have the following lemma.

\begin{lemm} \label{key 2 hyper}
Let $(M,g)$ be a non-negative biminimal hypersurface $($that is, a biminimal hypersurface with $\lambda \geq0$ $)$ in a Riemannian manifold with non-positive Ricci curvature. 
Then, the following inequality for $|{\bf H}|^2$ holds 
\begin{equation}\label{key 2} 
\Delta |{\bf H}|^2 \geq 2m~|{\bf H}|^4. 
\end{equation} 
\end{lemm}

\begin{proof} 
Since $\sum^m_{i=1}\langle R^N({\bf H},d\p(e_i))d\p(e_i),{\bf H}\rangle=H^2{\rm Ric}^N(\xi,\xi)\leq 0,$ the same argument as in Lemma \ref{key lemma} shows this lemma.
\end{proof} 

By using this lemma, we obtain the following results.

\begin{prop}
Let $(M,g)$ be a compact non-negative biminimal hypersurface in a Riemannian manifold with non-positive Ricci curvature. 
Then, $M$ is minimal.
\end{prop}

\begin{proof}
 Applying the standard maximum principle to the elliptic inequality (\ref{key 2}), 
we have that ${\bf H} = 0$ on $M$. 
\end{proof}

\begin{prop}
Let $(M,g)$ be a non-negative biminimal hypersurface in a Riemannian manifold with non-positive Ricci curvature. 
If the mean curvature is constant, then $M$ is minimal.
\end{prop}

\begin{proof}
Since $\Delta |{\bf H}|^2=0$, by using $(\ref{key 2})$, we obtain the proposition.  
\end{proof}

\begin{thm}\label{main theorem 2}
Let $(N,h)$ be a complete Riemannian manifold with non-positive Ricci curvature.
 Assume that the sectional curvature $K^N$ has a 
polynomial growth bound of order less than $2$ from below.
Then, any non-negative biminimal properly immersed hypersurface in $N$ is minimal.
 \end{thm}

\begin{proof}
By using Lemma $\ref{key 2 hyper}$, we obtain the inequality $(\ref{key 2})$.
Therefore, by using Lemma $\ref{proper sub mfd minimal}$, we obtain the theorem.
\end{proof}

Since every biharmonic submanifold is free biminimal one, we obtain the following result.

\begin{cor}
Let $(N,h)$ be a complete Riemannian manifold with non-positive Ricci curvature.
 Assume that the sectional curvature $K^N$ has a 
polynomial growth bound of order less than $2$ from below.
Then, any biharmonic properly immersed hypersurface in $N$ is minimal.
\end{cor}

By using Lemma $\ref{Cheng-Yau}$ and Lemma $\ref{key 2 hyper}$, we also obtain the following result.

\begin{prop}\label{Ric below 2}
Let $(M,g)$ be a complete non-negative biminimal hypersurface in a Riemannian manifold with
 non-positive Ricci curvature.
 If the Ricci curvature of $M$ is bounded from below,
 then $M$ is minimal.
 \end{prop}

Since every biharmonic submanifold is free biminimal one, we obtain the following result.

\begin{cor}
Let $(M,g)$ be a complete biharmonic hypersurface in a Riemannian manifold with
 non-positive Ricci curvature.
 If the Ricci curvature of $M$ is bounded from below,
 then $M$ is minimal.
\end{cor}

\vspace{10pt}

\noindent 
{\bf Acknowledgements.} 
The author would like to express his gratitude to
Professor Kazuo Akutagawa and Professor Hajime Urakawa for valuable suggestions and helpful discussions.

%\begin{acknowledgements}
%\end{acknowledgements}


\begin{thebibliography}{99} 
\bibitem{ka1}
K.~Akutagawa,
{\em On Spacelike Hypersurfaces with Constant Mean Curvature in the de Sitter Space},
Math. Z. {\bf 196}, (1987), 13--19.
\bibitem{kasm1}
 K.~Akutagawa and S.~Maeta,
 {\em Biharmonic properly immersed submanifolds in the Euclidean spaces},
 arXiv:1106.3222 [math DG]. 
 \bibitem{absmco1}
 A. Balmus, S. Montaldo and C. Oniciuc,
  {\em Classification results for biharmonic submanifolds in spheres},
   Israel J. Math. {\bf 168}, (2008), 201--220.
 \bibitem{absmco2}
 A. Balmus, S. Montaldo and C. Oniciuc,
 {\em Biharmonic hypersurfaces in $4$-dimensional space forms},
 Math. Nachr. {\bf 283}, (2010), 1696--1705.
 \bibitem{rcsmco1}
 R.~Caddeo, S.~Montaldo and C.~Oniciuc,
 {\em Biharmonic submanifolds of $S^3$},
 Internat. J. Math. {\bf 12}, (2001), 867--876.
 \bibitem{rcsmco2}
 R.~Caddeo, S.~Montaldo and C.~Oniciuc,
 {\em Biharmonic submanifolds in spheres},
 Israel J. Math. {\bf 130}, (2002), 109--123.
 \bibitem{rcsmpp1}
R.~Caddeo, S.~Montaldo and P.~Piu,
 {\em On biharmonic maps},
 Contemp. Math. {\bf 288}, (2001), 286--290.
 \bibitem{Calabi}
 E. Calabi,
 {\em An extension of E. Hopf's maximum principle with an application to Riemannian geometry}, 
 Duke Math. J. {\bf 25}, (1958), 45--56.
\bibitem {Chen} B.-Y.~Chen, 
    {\it Some open problems and conjectures on submanifolds of finite type}, 
    Michigan State University, (1988~version).  
%\bibitem {Chen-Ishikawa-1} B.-Y.~Chen and S.~Ishikawa, 
%    {\it Biharmonic surfaces in pseudo-Euclidean spaces}, 
%    Memoirs Fac.\ Sci., Kyushu Univ., Ser.~A {\bf 45}, (1991), 323--347. 
\bibitem {Chen-Ishikawa-2} B.-Y.~Chen and S.~Ishikawa, 
    {\it Biharmonic pseudo-Riemannian submanifolds in pseudo-Euclidean spaces}, 
    Kyushu J.\ Math.\ {\bf 52}, (1998), 167--185. 
\bibitem{Cheng-Yau2}
    S.-Y.~Cheng and S.-T.~Yau,
    {\em Differential equations on Riemannian manifolds and their geometric applications},
    Comm. Pure Appl. Math. {\bf 28}, (1975), 333--354.
\bibitem {Cheng-Yau} S.-Y.~Cheng and S.-T.~Yau, 
    {\it Maximal space-like hypersurfaces in the Lorentz-Minkowski spaces}, 
    Ann.\ of Math.\ {\bf 104}, (1976), 407--419. 
\bibitem{Colding-Minicozzi}
T.H.~Colding and W.P.~Minicozii II,
 {\em A Course in Minimal Surfaces},
 Graduate Studies in Mathematics {\bf 121}, Amer. Math. Soc., (2011). 
\bibitem {Leuven} F.~Defever, 
    {\it Hypersurfaces of} $\mathbb{E}^4$ {\it with harmonic mean curvature vector}, 
     Math.\ Nachr.\ {\bf 196}, (1998), 61--69.  
\bibitem {Dimi} I.~Dimitri\'{c}, 
    {\it Submanifolds of } $\mathbb{E}^m$ {\it with harmonic mean curvature vector}, 
    Bull.\ Inst.\ Math.\ Acad.\ Sinica {\bf 20}, (1992), 53--65. 
\bibitem{jell1}
 J. Eells and L. Lemaire, 
 {\em Selected topics in harmonic maps},
 CBMS, {\bf 50}, Amer. Math. Soc, (1983).
\bibitem {Hasanis-Vlachos} Th.~Hasanis and Th.~Vlachos, 
    {\it Hypersurfaces in} $\mathbb{E}^4$ {\it with harmonic mean curvature vector field}, 
    Math.\ Nachr.\ {\bf 172}, (1995), 145--169. 
\bibitem{ji2}
 J.~Inoguchi,
 {\em Biminimal submanifolds in contact $3$-manifolds},
 Balkan J. Geom. Appl. {\bf 12}, (2007), 56--67.
\bibitem{jg1}
 G. Y. Jiang, 
 {\em $2$-harmonic maps and their first and second variational formulas},
 Chinese Ann. Math. {\bf 7}A, (1986), 388--402;
  the English translation, Note di Mathematica {\bf 28}, (2008), 209--232.
 %\bibitem{jj1}
 %J.~Jost,
 %{\em Harmonic mapping between Riemannian manifolds},
 %Proceedings of the Centre for Mathematical Analysis 4, Australian National Univ., (1983).
% \bibitem{knt1}
% S. Kobayashi and K. Nomizu, 
% {\em Foundation of Differential Geometry, Vol. I, II},
%  John Wiley and Sons, New York (1963), (1969).
  %\bibitem{pfl1}
  %P.-F Leung,
  %{\em An estimate on the ricci curvature of a submanifold and some applications},
  %Proc. Amer. Math. Soc. {\bf 114}, 1051-1061 (1992).
 \bibitem{elsm1}
 E.~Loubeau and S.~Montaldo,
 {\em Biminimal immersion}, 
 Proc. Edinb. Math. Soc. {\bf 51}, (2008), 421--437.
 \bibitem{sm2}
 S.~Maeta,
 {k-harmonic maps into a Riemannian manifold with constant sectional curvature},
 Proc. Amer. Math. Soc. {\bf 140}, (2012), 1835--1847.
 \bibitem{sm3}
 S.~Maeta,
 {Polyharmonic submanifolds in Euclidean spaces},
 Balk. J. Geom. Appl. {\bf 17} No~1, (2012), 70--77.
 \bibitem{sm7}
 S.~Maeta,
 {\em Biminimal properly immersed submanifolds in the Euclidean spaces},
to appear in J. Geom. Phys.
\bibitem{smhu1}
   S.~Maeta and H.~Urakawa,
   {\em Biharmonic Lagrangian submanifolds in K\"ahler manifolds},
  arXiv:1203.4092 [math.DG].
\bibitem {N-U-1} N.~Nakauchi and H.~Urakawa, 
    {\it Biharmonic hypersurfaces in a Riemannian manifold with non-positive Ricci curvature}, 
    Ann.\ Global Anal.\ Geom. {\bf 40}, (2011), 125--131.
\bibitem {N-U-2} N.~Nakauchi and H.~Urakawa, 
    {\it Biharmonic submanifolds in a Riemannian manifold with non-positive curvature}, 
    to appear in Results. Math. 
%\bibitem{sn1}
%   S.~Nishikawa,
%   {\em On maximal spacelike hypersurfaces in a Lorenzian manifold},
%   Nagoya Math. J. {\bf 95}, (1984), 117--124.
\bibitem{ylolt1}
Y.-L.~Ou and L.~Tang,
 {\em The generalized Chen's conjecture on biharmonic submanifolds is false},
 arXiv:1006.1838 [math.DG].
% \bibitem{ylolt2}
%Y.-L.~Ou and L.~Tang,
%{\em Biharmonic hypersurfaces in a conformally flat space},
%arXiv:1204.5550 [mathDG].
% \bibitem{sp1}
% X.~Pennec,
% {\em Intrinsic Statistics on Riemannian Manifolds: Basic Tools for Geometric Measurements},
% J. Math. Imaging and Vision, {\bf 25}, (2006), 127--154.
 \bibitem{t.sakai1}
 T.~Sakai,
 {\em Riemannian geometry}, Transl. Math. Monographs {\bf 149}, Amer. Math. Soc., (1996).
\bibitem{ts1}
 T. Sasahara,
{\em  A classification result for biminimal Lagrangian surfaces in complex space forms,}
 J. Geom. Phys. {\bf 60}, (2010), 884--895.
\bibitem {Yau} S.-T.~Yau, 
    {\it Harmonic functions on complete Riemannian manifolds}, 
     Comm.\ Pure Appl.\ Math.\ {\bf 28}, (1975), 201--228. 
\end{thebibliography}


\qquad\\
\qquad\\

%%%%%%%%%%%%%%%%%%%%%%%%%%%%%%%%%%%%%%
\bibliographystyle{amsbook}

\vspace{10pt}

\end{document}